\documentclass[12pt]{article}
\usepackage[T1]{fontenc}
\usepackage[utf8]{inputenc}

\usepackage{amsmath,amsthm}     
\usepackage{graphicx}     
\usepackage{hyperref} 
\usepackage{url}
\usepackage{amsfonts} 
\usepackage{geometry}
\usepackage{multicol}

\theoremstyle{definition}
\newtheorem{theorem}{Theorem}
\newtheorem{lemma}[theorem]{Lemma}

\allowdisplaybreaks

\makeatletter
\@addtoreset{footnote}{page}
\makeatother

\begin{document}

\title{\vspace{-2cm}An Idempotent Cryptarithm}

\author{Samer Seraj\\
Existsforall Academy\\
samer\_seraj@outlook.com}

\maketitle

\begin{abstract}
    Notice that the square of $9376$ is $87909376$ which has as its rightmost four digits $9376$. To generalize this remarkable fact, we show that, for each integer $n\ge 2$, there exists at least one and at most two positive integers $x$ with exactly $n$-digits in base-$10$ (meaning the leftmost or $n^{\text{th}}$ digit from the right is non-zero) such that squaring the integer results in an integer whose rightmost $n$ digits form the integer $x$. We then generalize the argument to prove that, in an arbitrary number base $B\ge 2$ with exactly $m$ distinct prime factors, an upper bound is $2^m -2$ and a lower bound is $2^{m-1}-1$ for the number of such $n$-digit positive integers. For $n=1$, there are exactly $2^m -1$ solutions, including $1$ and excluding $0$.
\end{abstract}

\noindent In Kordemsky's book, \textit{The Moscow Puzzles}, the following multiplicative cryptarithm is posed as Problem 272 (G), where ``... digits are represented by letters and asterisks. Identical letters stand for identical digits, different letters stand for different digits. An asterisk stands for any digit.''
$$\begin{array}{cccccccc}
    &&&& A & T & O & M\\
    &&&\times & A & T & O & M\\ \hline
    *&*&*& * & A & T & O & M
\end{array}$$
Only a hint is provided in the solutions section: ``What is the last digit of $M\times M$? What four digits have this property? Which two can be immediately eliminated? Then consider $OM\times OM$, and so on. Can you prove, on your own, that $ATOM=9376$ is the only possible solution?''

In terms of modular arithmetic, we want to find the (apparently unique) $4$-digit positive integer $x$ such that $$x^2 \equiv x \pmod{10^4}.$$ It is natural to extend the question by fixing any positive integer $n$ and number base $B\ge 2$ and asking for all non-negative solutions $x$ to the idempotence congruence $$x^2 \equiv x \pmod{B^n}$$ that have $n$ or fewer digits in base-$B$. The last congruence is equivalent to saying that that the rightmost $n$ digits of $x^2$ in base-$B$ form the non-negative integer $x$ in base-$B$ (possibly with $0$'s padded on the left end) because higher order digits disappear modulo $B^n$. Moreover, in line with Kordemsky's implicit condition that $A\ne 0$ in $ATOM$, we can wonder how many of these solutions have \textit{exactly} $n$ digits, meaning the leftmost of the $n$ digits is non-zero. It is with this question that we will be occupied, first in base-$10$ and later for general bases $B\ge 2$. The quest for idempotent elements is not new, but the author is unaware of a previous result about finding a lower bound on the number of idempotent elements with exactly $n$ digits.

\section{Structure of solutions in base-10}

For $n=1$, it is easy to check all ten residues modulo $10$ and find that the only base-$10$ digits that satisfy $$x^2\equiv x\pmod{10^n}$$ are $x=0,1,5,6$. Now we tackle $n\ge 2$.

\begin{lemma}\label{lemma1}
For each integer $n\ge 2$, a multiplicative inverse of $5$ modulo $2^n$ is $5^{2^{n-2}-1}$.
\end{lemma}

\begin{proof}
We will show by induction on $n\ge 2$ that
$$5^{2^{n-2}} \equiv 1\pmod{2^n}.$$
As a result, it will hold that
$$\begin{array}{rclr}
    5\cdot 5^{2^{n-2}-1} &\equiv & 1 & \pmod{2^n},\\
    5^{-1} & \equiv & 5^{2^{n-2}-1} & \pmod{2^n}.
\end{array}$$
The base case $n=2$ holds because $$5^{2^{2-2}} = 5 \equiv 1 \pmod{2^2}.$$ For the induction hypothesis, suppose $$5^{2^{n-2}} \equiv 1 \pmod{2^n}$$ for some integer $n\ge 2$, which is equivalent to assuming that $$5^{2^{n-2}} = 1 + x_n 2^n$$ for some integer $x_n$. Squaring the equation yields
\begin{align*}
    5^{2^{n-1}} &= 1 + x_n 2^{n+1} + x_n^2 2^{2n}\\
    &= 1 + x_n(1+x_n 2^{n-1}) 2^{n+1}\\
    &\equiv 1\pmod{2^{n+1}}.
\end{align*}
This completes the inductive step and therefore the induction argument.
\end{proof}

\begin{theorem}\label{theorem2}
For integers $n\ge 2$, there are exactly four distinct non-negative solutions to the congruence $$x^2\equiv x\pmod{10^n}$$ that are less than $10^n$. They are $0,1$ and the remainders of
\begin{align*}
    a_n &= 5^{n 2^{n-2}},\\
    b_n &= 1 - a_n
\end{align*}
modulo $10^n$. We will call the former two the ``trivial solutions'' and the latter two the ``non-trivial solutions'' $r_n$ (corresponding to $a_n$) and $s_n$ (corresponding to $b_n$).
\end{theorem}

\begin{proof}
First we note that $x$ is a solution to $$x^2 \equiv x\pmod{10^n}$$ if and only if $$x(x-1)\equiv 0\pmod{2^n 5^n}.$$ This holds if and only if $2^n$ and $5^n$ both divide $x(x-1)$ because $2^n$ and $5^n$ are coprime. Since $x$ and $x-1$ are coprime, all of the factors of $2$ belong to one of them and all of the factors of $5$ belong to one of them. This leads to four pairs of congruences:
\begin{align*}
    x &\equiv t_1 \pmod{2^n},\\
    x &\equiv t_2 \pmod{5^n},
\end{align*}
where $(t_1,t_2)$ can be $(0,0),(1,1),(1,0),(0,1)$. By the Chinese remainder theorem, each of the four systems gives rise to a unique solution modulo $10^n$. The $t_1=t_2=0$ case leads to the solution $x=0$, and the $t_1=t_2=1$ case leads to the solution $x=1$. For the other two take more effort to solve:
\begin{itemize}
    \item If $t_1=1$ and $t_2=0$, then there exists an integer $y$ such that $x=5^n y$, so $$5^n y = x \equiv 1\pmod{2^n}.$$ By Lemma \ref{lemma1}, we know the multiplicative inverse of $5$ modulo $2^n$, so $$y\equiv (5^{-1})^n \equiv \left(5^{2^{n-2}-1}\right)^n\pmod{2^n}.$$ So there exists an integer $z$ such that
    \begin{align*}
        x &= 5^n y= 5^n\left[\left(5^{2^{n-2}-1}\right)^n+2^n z\right]=5^{n2^{n-2}}+10^n z,\\
        x &\equiv 5^{n2^{n-2}}\pmod{10^n}.
    \end{align*}
    \item If $t_1=0$ and $t_2=1$, then we can avoid repeating the above computations as follows. If the solution here is $b_n$ and the solution in the previous case is called $a_n$, then the two pairs of congruences
    \begin{align*}
        a_n &\equiv 1 \pmod{2^n},\\
        a_n &\equiv 0 \pmod{5^n}
    \end{align*}
    and
    \begin{align*}
        b_n &\equiv 0 \pmod{2^n},\\
        b_n &\equiv 1 \pmod{5^n}
    \end{align*}
    lead to
    $$
    \begin{cases}
        a_n+b_n \equiv 1 \pmod{2^n}\\
        a_n+b_n \equiv 1 \pmod{5^n}
    \end{cases}
    \implies a_n+b_n \equiv 1\pmod{10^n}.
    $$
    As a result, we can choose $b_n =1-a_n$.
\end{itemize}
Note that all four solutions leave distinct remainders modulo $10^n$ because they satisfy distinct systems of two congruences with the same constituent moduli $2^n$ and $5^n$. In other words, if any two of these remainders were the same modulo $10^n$, then they would satisfy the same pair of congruences modulo $2^n$ and $5^n$, which would be untrue.
\end{proof}

\begin{lemma}\label{lemma3}
For each integer $n\ge 2$, let $r_n$ and $s_n$ be as in Theorem \ref{theorem2}; let $r_1=5$ and $s_1=6$. Then for each pair of positive integers $n>k$,
\begin{align*}
    r_n &\equiv r_k \pmod{10^k},\\
    s_n &\equiv s_k \pmod{10^k}.
\end{align*}
\end{lemma}

\begin{proof}
Let $n>k$ be a pair of positive integers. If $x$ is an integer such that the congruences
\begin{align*}
    x &\equiv t_1 \pmod{2^n},\\
    x &\equiv t_2 \pmod{5^n},
\end{align*}
hold, then the congruences
\begin{align*}
    x &\equiv t_1 \pmod{2^k},\\
    x &\equiv t_2 \pmod{5^k}
\end{align*}
also hold since $2^k \mid 2^n$ and $5^k \mid 5^n$. By the Chinese remainder theorem, the first pair of congruences has the unique solution $0\le r_n<10^n$ for $(t_1,t_2)=(1,0)$ and the unique solution $0\le s_n< 10^n$ for for $(t_1,t_2)=(0,1)$. Since $r_k$ and $s_k$ uniquely solve the second pair of congruences modulo $10^k$ for the respective $(t_1,t_2)$, we find that $r_n \equiv r_k \pmod{10^k}$ and $s_n \equiv s_k \pmod{10^k}$.
\end{proof}

Lemma \ref{lemma3} explains why Table 1 shows that each value of $r_n$ or $s_n$ in a column simply pads a digit (possibly $0$) on to the left of the previous value in the same class.

\begin{table}[h]\label{table1}
\begin{center}
\begin{tabular}{r|r|r}
    $n$  &        $r_n$ &        $s_n$ \\ \hline
    $1$  &          $5$ &          $6$ \\
    $2$  &         $25$ &         $76$ \\
    $3$  &        $625$ &        $376$ \\
    $4$  &       $0625$ &       $9376$ \\
    $5$  &      $90625$ &      $09376$ \\
    $6$  &     $890625$ &     $109376$ \\
    $7$  &    $2890625$ &    $7109376$ \\
    $8$  &   $12890625$ &   $87109376$ \\
    $9$  &  $212890625$ &  $787109376$ \\
    $10$ & $8212890625$ & $1787109376$
\end{tabular}
\end{center}
\caption{The two non-trivial solutions for the first ten positive integers $n$ in base-$10$}
\end{table}

Observe in Table 1 that all pairs of non-units digits of $r_n$ and $s_n$ in corresponding places add up to $9$. We will prove this fact and use it to prove Theorem \ref{theorem4}.

\begin{theorem}\label{theorem4}
Let $n$ be a positive integer. In the notation of Theorem \ref{theorem2} and Lemma \ref{lemma3}, let the non-trivial solution corresponding to $a_n$ be $r_n$ and the non-trivial solution corresponding to $b_n$ be $s_n$. At least one of $r_n$ or $s_n$ has a non-zero $n^{\text{th}}$ digit from the right, meaning at least one of $r_n$ or $s_n$ is an $n$-digit integer.
\end{theorem}

\begin{proof}
We will first prove that, for all positive integers $n$, $$r_n+s_n=10^n +1.$$ We know that the following two congruences hold:
\begin{align*}
    a_n &\equiv r_n \pmod{10^n},\\
    1-a_n= b_n &\equiv s_n \pmod{10^n}.
\end{align*}
Adding them leads to $$r_n+s_n-1\equiv 0\pmod{10^n}.$$
Since neither $r_n$ nor $s_n$ are either of the trivial solutions $0$ or $1$, they are each at least $2$, resulting in $$0<3\le r_n+s_n-1.$$ Moreover, $r_n$ and $s_n$ are both at most $n$-digit integers, so $$r_n+s_n-1\le (10^n -1)+(10^n -1)-1=2\cdot 10^n -3<2\cdot 10^n.$$ The only way that $r_n+s_n-1$ is a multiple of $10^n$ in the interval $(0,2\cdot 10^n)$ is if $$r_n+s_n=10^n +1.$$

So, for $n\ge 2$, we know that
\begin{align*}
    r_{n-1}+s_{n-1} &= 10^{n-1}+1,\\
    r_n+s_n &= 10^n +1.
\end{align*}
Then
\begin{align*}
    (r_n-r_{n-1})+(s_n-s_{n-1}) &= (r_n+s_n)-(r_{n-1}+s_{n-1}) \\
    &= (10^n +1)-(10^{n-1} +1)\\
    &= 9\cdot 10^{n-1}.
\end{align*}
This proves that the $n^{\text{th}}$ digits from the right of $r_n$ and $s_n$ add up to $9$ for $n\ge 2$, so at least one of them has to be non-zero.
\end{proof}

With the existence result established, it may be asked when there exists exactly one $n$-digit solution or exactly two $n$-digit solutions. In light of the observation about pairing digits in corresponding places in Theorem \ref{theorem4}, it holds that in order for there to be exactly one $n$-digit solution, one of $r_n$ or $s_n$ must have leftmost digit $0$ and the other must have leftmost digit $9$. However, computational results hint that this is a difficult property to predict. The author encourages the reader to explore this area.

\section{Bounds for the number of solutions in base-B}

Next, we will generalize the result to any number base. As Hilbert said, ``The art of doing mathematics consists in finding that special case which contains all the germs of generality.'' Indeed, the method for base-$10$ illuminates the path to a more general result as follows.

\begin{theorem}
Let $B\ge 2$ be an integer with exactly $m\ge 1$ distinct prime factors. Then, for each integer $n\ge 2$, there are at most $2^m -2$ and at least $2^{m-1}-1$ positive integers $x$ with exactly $n$ digits in base-$B$ (meaning the leftmost or $n^{\text{th}}$ digit from the right is non-zero) such that $$x^2 \equiv x \pmod{B^n}.$$ For $n=1$, there are exactly $2^m -1$ positive (single-digit) solutions.
\end{theorem}

\begin{proof}
Let the base be $$B = p_1^{e_1}p_2^{e_2}\cdots p_m^{e_m}$$ for distinct primes factors $p_i$ and positive integer multiplicities $e_i$. Using the fact that the prime powers $p_i^{e_i}$ are all coprime to each other and that $x$ and $x-1$ are coprime to each other, we find that $x$ is a solution to $$x^2 \equiv x \pmod {B^n}$$ if and only if a system of $m$ congruences of the following form is satisfied:
\begin{align*}
    x &\equiv 0 \text{ or } 1 \pmod{p_1^{ne_1}},\\
    x &\equiv 0 \text{ or } 1 \pmod{p_2^{ne_2}},\\
    &\vdots\\
    x &\equiv 0 \text{ or } 1 \pmod{p_m^{ne_m}}.
\end{align*}
Each right side can be $0$ or $1$, so there are $2^m$ possible systems of congruences. Each such system has a unique solution $0\le x<B^n$ by the Chinese remainder theorem. So there are at most $2^m$ solutions with $n$ digits. If all of the $t_i$ equal $0$ then the solution is $0$, and if all of the $t_i$ equal $1$ then the solution is $1$. We may omit the $x=0,1$ solutions (which have fewer than $n\ge 2$ digits) for an upper bound of $2^m -2$. Now we will use a pairing technique. Given a particular system of congruences of the above form, we can pair it with its twin which swaps the $0$'s for $1$'s and $1$'s for $0$'s on the right sides of the congruences. This produces $2^{m-1}-1$ non-trivial twins. If $a_n$ is a solution to one of the two systems and $b_n$ is a solution to the other, we find that $$a_n + b_n \equiv 1 \pmod{B^n}$$ by adding corresponding pairs of congruences from the two systems of congruences, and then combining all of the moduli. We will show that the remainder modulo $B^n$ of at least one of $a_n$ or $b_n$ has exactly $n$ digits in base-$B$.

Note that, if a particular instance of the above system, say
\begin{align*}
    x &\equiv t_1 \pmod{p_1^{ne_1}},\\
    x &\equiv t_2 \pmod{p_2^{ne_2}},\\
    &\vdots\\
    x &\equiv t_m \pmod{p_m^{ne_m}},
\end{align*}
is satisfied, then the system
\begin{align*}
    x &\equiv t_1 \pmod{p_1^{ke_1}},\\
    x &\equiv t_2 \pmod{p_2^{ke_2}},\\
    &\vdots\\
    x &\equiv t_m \pmod{p_m^{ke_m}}
\end{align*}
is also satisfied for positive integers $k<n$. This means that, similar to Lemma \ref{lemma3}, we can classify the solutions over all positive integers $n$ into $2^m$ classes (including the $0$-class and the $1$-class) corresponding to the $2^m$ different $m$-tuples $$(t_1,t_2,\ldots,t_m)\in \{0,1\}^m,$$ with each solution within a particular class padding a digit (possibly $0$) to the left of the previous solution in that class.

Let $r_n$ be the remainder of $a_n$ and $s_n$ be the remainder of $b_n$ modulo $B^n$. Then $$r_n+s_n \equiv a_n+b_n\equiv 1\pmod{B^n}.$$ Now we will follow the method shown in Theorem \ref{theorem4} of using inequalities to prove that $$r_n+s_n=B^n +1$$ for all positive integers $n$. Since we have already omitted the trivial twin solutions $0$ and $1$ from the possibilities, $r_n$ and $s_n$ are each at least $2$, so $$0<3=2+2-1\le r_n+s_n-1.$$ Moreover, $r_n$ and $s_n$ each have at most $n$ digits in base-$B$, so $$r_n+s_n -1 \le (B^n -1)+(B^n -1) -1 = 2\cdot B^n -3 < 2\cdot B^n.$$ The only way that $B^n\mid r_n+s_n-1$, with $r_n+s_n-1\in (0,2\cdot B^n)$ is if $$r_n+s_n=B^n +1,$$ as predicted.

So, for all integers $n\ge 2$, we can use the equations
\begin{align*}
    r_{n-1}+s_{n-1} &= B^{n-1}+1,\\
    r_n+s_n &= B^n +1
\end{align*}
to find that
\begin{align*}
    (r_n-r_{n-1})+(s_n-s_{n-1}) &= (r_n+s_n)-(r_{n-1}+s_{n-1}) \\
    &= (B^n +1)-(B^{n-1} +1)\\
    &= (B-1) B^{n-1}.
\end{align*}
Thus, for $n\ge 2$, the $n^{\text{th}}$ digit from the right of $r_n$ and the $n^{\text{th}}$ digit from the right of $s_n$ add up to $B-1>0$. So both of those digits cannot be $0$, proving that at least one of the solutions in this twin has exactly $n$ digits in base-$B$. There are $2^{m-1}-1$ of these non-trivial twins, which establishes the lower bound. 

Lastly, the one-digit positive solutions (for $n=1$) are all the solutions in the interval $(0,B)$ of the congruence $$x(x-1)\equiv 0\pmod {B}.$$ There are $2^m$ solutions by distributing the maximal prime power factors of $B$ across the coprime integers $x$ and $x-1$. We omit $0$ to get the exact number $2^m -1$.
\end{proof}

\begin{center}
    \textbf{Acknowledgement}
\end{center}

The author thanks Dr. Edward Barbeau for providing helpful suggestions regarding simplifying the expression for the second non-trivial solution in Theorem \ref{theorem2}. This also resulted in a simplification of the proof of Theorem \ref{theorem4}.

\end{document}